\documentclass[10pt,a4paper]{amsart}
\usepackage{amsthm,amsfonts,amsmath,amssymb,latexsym}
\usepackage{epsfig,graphics,color}
\usepackage[matrix,arrow,curve]{xy}
\usepackage{hyperref}
\usepackage{mathrsfs}
\usepackage{amsmath}
\usepackage{amsfonts}
\usepackage{amssymb}
\usepackage{graphicx}
\usepackage{dashbox}
\usepackage{epigraph}
\usepackage{verbatim}
\newtheorem{theorem}{Theorem}
\newtheorem{corollary}{Corollary}
\newtheorem{lemma}{Lemma}
\newtheorem{proposition}{Proposition}
\newtheorem{definition}{Definition}

\theoremstyle{definition}
\newtheorem{remark}{Remark}
\newtheorem{example}{Example}

\newcommand{\cC}{\mathcal{C}}

\newcommand{\cH}{\mathcal{H}}

\newcommand{\cL}{\mathcal{L}}

\newcommand{\cT}{\mathcal{T}}

\newcommand{\C}{{\mathbb{C}}}

\newcommand{\R}{{\mathbb{R}}}

\newcommand{\Z}{{\mathbb{Z}}}
\newcommand{\Aut}{\mathrm{ Aut}}          
\newcommand{\Hom}{\mathrm{ Hom}}          

\newcommand{\p}{{\partial}}

\begin{document}
\title[Local systems and finiteness of the Hofer-Zehnder capacity]{Local systems on the free loop space and finiteness of the Hofer-Zehnder capacity}
\author{Peter Albers}
\address{Mathematisches Institut, Westf\"alische Wilhelms-Universit\"at
M\"unster, \ Einsteinstrasse 62, D-48149 M\"unster, Germany}
\email{peter.albers@uni-muenster.de}
\thanks{P.A. partially funded by SFB 878}
\author{Urs Frauenfelder}
\address{Universit\"at Augsburg, Universit\"atsstrasse 14, D-86159 Augsburg, Germany}
\email{urs.frauenfelder@math.uni-augsburg.de}
\author{Alexandru Oancea}
\address{Sorbonne Universit\'es, UPMC Univ Paris 06, UMR 7586, Institut de Math\'ematiques de Jussieu-Paris Rive Gauche, Case 247, 4 place Jussieu, F-75005, Paris, France}
\email{alexandru.oancea@imj-prg.fr}
\thanks{A.O. partially funded by the European Research Council, StG-259118-STEIN}


\date{\today}

\maketitle

\setlength{\epigraphwidth}{5.5cm}

\epigraph{{\it ``Local coefficients bring an extra level of complication that one tries to avoid
whenever possible."}

\flushright{--- Allen Hatcher, {\it Algebraic Topology}}}


\begin{abstract}
In this article we examine under which conditions symplectic homology with local coefficients of a unit disk bundle $D^*M$ vanishes. For instance this is the case if the Hurewicz map $\pi_2(M)\to H_2(M;\Z)$ is nonzero. As an application we prove finiteness of the $\pi_1$-sensitive Hofer-Zehnder capacity of unit disk bundles in these cases. We also prove uniruledness for such cotangent bundles. Moreover, we find an obstruction to the existence of $H$-space structures on general topological spaces, formulated in terms of local systems. 
\end{abstract}


\section*{Introduction}

Symplectic geometry has its origins in Hamiltonian dynamics, and arguably one of the most important classes of symplectic manifolds is that of cotangent bundles. A manifold $M$ represents the configuration space and the cotangent bundle $T^*M$ represents the phase space of a physical system.  The Hofer-Zehnder capacity was introduced in \cite{Hofer_Zehnder_A_new_capacity_for_symplectic_manifolds,Hofer_Zehnder_Book}. Concerning the history and the fundamental importance of symplectic capacities for our understanding of symplectic geometry we refer to the article~\cite{CHLS}, which includes in particular a comprehensive list of references. As a matter of fact there is a whole lattice of Hofer-Zehnder capacities indexed by the lattice of collections of free homotopy classes of loops, and the largest of these capacities is the so-called \emph{$\pi_1$-sensitive Hofer-Zehnder capacity}, which takes into account only contractible periodic orbits. We refer to the smallest of these capacities, which takes into account all free homotopy classes of loops, as the \emph{Hofer-Zehnder capacity}. A striking consequence of the finiteness of the ($\pi_1$-sensitive) Hofer-Zehnder capacity is an almost existence theorem for (contractible) periodic orbits close to a regular energy level, see~\S\ref{sec:HZ}, Theorem~\ref{thm:almost_existence}. Unfortunately, to the authors' knowledge it is unknown whether the Hofer-Zehnder capacities are finite for all unit disk bundles $D^*M$ over closed manifolds. The $\pi_1$-sensitive Hofer-Zehnder capacity need actually not be finite for every unit disk bundle: it is infinite if the closed manifold $M$ admits a Riemannian metric without contractible closed geodesics, for instance for the torus $M=T^n$ or hyperbolic manifolds, see also the discussion in~\S\ref{sec:HZ}.

Recent progress was made by Kei Irie~\cite{Irie-HZ}, who proved that the Hofer-Zehnder capacity is finite provided the manifold $M$ carries an $S^1$-action with non-contractible orbits, which covers of course the case of the torus $T^n$. Note that finiteness of the Hofer-Zehnder capacity of $D^*T^n$ follows alternatively from the finiteness of the Hofer-Zehnder capacity of standard symplectic balls and the fact that $T^n$ admits a Lagrangian embedding in $\R^{2n}$, using the Weinstein neighborhood theorem. The same argument works for any closed manifold $M$ that admits a Lagrangian embedding in $\R^{2n}$, see~\cite[\S1.4]{Irie-HZ} and the references therein.

Our main result is a proof of the finiteness of the $\pi_1$-sensitive Hofer-Zehnder capacity for a class of closed manifolds that is essentially disjoint from the above. We say that a closed connected manifold $M$ satisfies \emph{condition~(C)} if one of the following two conditions holds with $G=\pi_1(M)$:
\renewcommand{\theenumi}{\roman{enumi}}
\begin{enumerate}
\item There exists a prime number $\ell\ge 2$ such that the map $H^2(G;\Z/\ell)\to H^2(M;\Z/\ell)$ is not surjective.
\item There exists a prime number $\ell\ge 2$ such that the map $H^3(G;\Z/\ell)\to H^3(M;\Z/\ell)$ is not injective. 
\end{enumerate}
The maps on cohomology are induced by the classifying map $M\to BG$ for the $G$-principal bundle given by the universal cover $\tilde M\to M$.

\noindent {\bf Remark.}  As shown in \S\ref{sec:loc-sys-loop-space}, Proposition~\ref{prop:pi2}, condition (i) is equivalent to:
{\it 
\begin{enumerate}
\item[(i')] The Hurewicz map $\pi_2(M)\to H_2(M;\Z)$ is nonzero.
\end{enumerate}
}

\noindent {\bf Theorem.} {\it Let $M$ satisfy condition (C). Then the $\pi_1$-sensitive Hofer-Zehnder capacity of $D^*M$ is finite.
}

Condition~(i') holds for example if $M$ is a simply-connected closed manifold with nonzero integral second homology, for example the sphere $S^2$. This is of particular interest with respect to the planar restricted $3$-body problem~\cite{AFKP}. In this example, the bound on the Hofer-Zehnder capacity is given by the regularized period of the doubly-covered retrograde periodic orbit. Note that simply-connected manifolds are not covered by Irie's theorem~\cite{Irie-HZ}, and also do not admit Lagrangian embeddings in $\R^{2n}$ as proved by Gromov. A relevant class of closed manifolds for which the theorem does \emph{not} apply is that of $K(\pi,1)$'s, and this is related to the fact that the $\pi_1$-sensitive Hofer-Zehnder capacity can be infinite in such cases, as shown by the examples of tori and hyperbolic manifolds.  

Regarding manifolds which are not simply connected, one easy way to prove finiteness of the $\pi_1$-sensitive Hofer-Zehnder capacity is to pass to a finite cover. Indeed, if the unit disk bundle of a finite cover has finite capacity, so does the unit disk bundle of the manifold itself. One relevant example is that of $\R P^2$. In this way, we obtain in a straightforward way finiteness of the $\pi_1$-sensitive Hofer-Zehnder capacity of closed manifolds with finite fundamental group and non-trivial second homotopy group. However, our theorem also applies to a large class of manifolds with infinite fundamental group, for example  complex blow-ups of any closed manifold of even dimension $\ge 4$. One such example is that of the blown-up torus $T^4$. Note that in this case the second homotopy group is not even finitely generated. 

Following in the footsteps of Biolley~\cite{Biolley-thesis,Biolley-arxiv}, we also prove in~\S\ref{sec:uniruledness} that the cotangent bundle $T^*M$ is uniruled if condition (C) holds. As an amusing byproduct we obtain in~\S\ref{sec:products}, Proposition~\ref{prop:RPn} an obstruction to the existence of $H$-space structures.

Note that we do not assume the manifold $M$ to be orientable, though we could do so to obtain finiteness of the Hofer-Zehnder capacity by passing to the orientation double cover. We are able to analyze the non-orientable case thanks to the work of Abouzaid~\cite{Abouzaid-cotangent}, which explains the isomorphism between symplectic homology and the homology of the free loop space also for non-orientable manifolds. See also the discussion below. 

The key ingredient in the proof is the fact that, under condition (C), there exists a local system of coefficients $\cL$ with fiber $\C$ on the free loop space $LD^*M$, which restricts to a trivial local system on the space of constant loops, and such that symplectic homology of $D^*M$ with coefficients in $\cL$ vanishes. On the other hand, Irie proved~\cite[Corollary~3.5]{Irie-HZ} that the vanishing of the symplectic homology with constant $\Z/2$-coefficients implies finiteness of the $\pi_1$-sensitive Hofer-Zehnder capacity for any Liouville domain. His result adapts to the above setup, and leads to a proof of our main theorem. Uniruledness is proved using the same vanishing result and a symplectic field theory neck-stretching argument. 

The use of local coefficients in order to force the vanishing of symplectic homology first appeared in the work of Ritter~\cite{Ritter09}, see also~\cite{Ritter-ALE,Ritter}. In the case of cotangent bundles, vanishing is obtained in~\cite{Ritter09} for a Novikov-type $\Z$-local system under a finite type assumption. Our use of local coefficients in order to force the vanishing of symplectic homology is inspired by a remark in an unpublished note of Seidel, who proved that the homology of the free loop space of $\mathbb CP^2$ with coefficients in a $\Z/2$-local system that is nontrivial on the based loop space must vanish~\cite{Seidel-unpublished}. Seidel's proof uses the Pontryagin ring structure on the homology of the based loop space, and we use the same idea in a more general setting. 

That local systems of coefficients on the free loop space should -- and actually do -- play a role in symplectic topology became clear after Ritter's work and also Kragh's observation that Viterbo's isomorphism~\cite{Viterbo-cotangent,AS,SW,Abouzaid-cotangent} between the symplectic homology of the cotangent bundle of a closed manifold and the homology of its free loop space holds with coefficients more general than $\Z/2$ only after twisting one of the homology groups by a particular transgressive local system determined by the second Stiefel-Whitney class, see~\cite{Abouzaid-cotangent,AS-corrigendum}. 
Twisted string topology operations have previously appeared in~\cite{Ritter}.
Local systems with fiber $\C$ and holonomy in $U(1)$ also play a key role in mirror symmetry. 

The deep reason why our method currently works only for cotangent bundles rather than for general Weinstein domains is that this is the only instance in which we have a topological interpretation for symplectic homology, as the homology of the free loop space. But there should be much more general instances in which one should be able to find such ``killer" local systems, see the discussion in~\S\ref{sec:open_questions}. 

In this paper we consider rank $1$ local systems with fiber $\C$ and holonomy in $\Z/\ell$ for some prime number $\ell$. Here $\Z/\ell$ is identified with the multiplicative subgroup $\{1,\zeta,\zeta^2,\dots,\zeta^{\ell-1}\}\subset \C^*$ determined by a primitive $\ell$-th root of unity $\zeta$.

\noindent {\bf Convention}. For any topological space $X$ we denote $H_\cdot(X)$ and $H^\cdot(X)$ homology, respectively cohomology with arbitrary and unspecified constant coefficients.

\noindent {\bf Acknowledgements}. The third author wishes to thank Nancy Hingston for inspiring comments, and the College of New Jersey for hospitality during the summer 2015.


\section{Local systems on loop spaces and Pontryagin product} \label{sec:local_systems}

Given a topological space $X$ we denote by $ct$ any constant path in $X$ and, given two paths $\alpha,\beta$ in $X$ we denote by $\alpha\cdot \beta$ their concatenation, defined by first following $\alpha$ and then following $\beta$, with double speed. 

\subsection{Preliminaries on local systems} 

\subsubsection{Classification}

\begin{definition} \label{defi:local-system}
Let $X$ be a path-connected topological space 
which admits a universal cover 
and let $R$ be a commutative ring with unit. A \emph{local system of coefficients on $X$} consists of the following data: 
\begin{itemize}
\item a collection $M_x$ of $R$-modules, one for each point $x\in X$, called \emph{the fibers} of the local system;
\item a collection of isomorphisms of $R$-modules $\tau_{[\alpha]}:M_x\to M_y$, one for each pair $(x,y)\in X\times X$ and each homotopy class $[\alpha]$ of paths from $x$ to $y$ with fixed endpoints, subject to the conditions 
$$
\tau_{[ct]}=\operatorname{Id}
$$ 
and 
$$
\tau_{[\alpha]\cdot[\beta]} = \tau_{[\beta]}\tau_{[\alpha]}.
$$
\end{itemize} 
We say that a local system has \emph{rank $k$} if all its fibers are free $R$-modules of rank $k$.  
\end{definition}

The above data provides in particular for any $x\in X$ a group homomorphism 
$$
\rho'_x\in\Hom(\pi_1(X;x),\Aut(M_x)),
$$
which we call \emph{holonomy representation at $x$}. The holonomy representations at two distinct points $x,y\in X$ are related as follows. For any fixed endpoint homotopy class $[\alpha]$ of paths from $x$ to $y$ there are natural isomorphisms $\phi_{[\alpha]}:\pi_1(X;x)\to \pi_1(X;y)$, $[\gamma]\mapsto [\alpha^{-1}]\cdot[\gamma]\cdot[\alpha]$, and $\Phi_{[\alpha]}:\Aut(M_x)\to \Aut(M_y)$, $\psi\mapsto \tau_{[\alpha]}\psi\tau_{[\alpha^{-1}]}$. Then $\rho'_x=\Phi_{[\alpha]}^{-1}\rho'_y\phi_{[\alpha]}$. 

Let $M$ be an $R$-module which is isomorphic to any of the fibers $M_x$ of a local system on $X$. For any choice of a basepoint $x\in X$ and for any choice of isomorphism $M_x\simeq M$, we obtain an element $\tilde \rho_x\in\Hom(\pi_1(X;x),\Aut(M))$. This element is well-defined up to inner automorphisms in the target and we denote its equivalence class by 
$$
\rho_x\in \Hom(\pi_1(X;x),\Aut(M))/\Aut(M).
$$ 
For any choice of homotopy class $[\alpha]$ with fixed endpoints connecting $x$ and $y$ we have $\rho_x=\rho_y\phi_{[\alpha]}$. In order to emphasize the fact that this relation holds for any class $[\alpha]$, we shall suppress from now on the mention of the basepoint $x$ and write 
$$
\rho \in\Hom(\pi_1(X),G)/G,\qquad G=\Aut(M).
$$
We call $\rho$ \emph{the holonomy representation}, and we speak of a \emph{local system of coefficients with fiber $M$}. The holonomy representation determines uniquely the local system up to isomorphism~\cite[Theorem~1]{Steenrod}. 

Any representative $\rho'_{x_0}$ of the holonomy representation determines the bundle $\tilde X\times_{\pi_1(X;x_0)}ÊM$, where $\pi_1(X;x_0)$ acts on $M$ via $\rho'_{x_0}$ and on the universal cover $\tilde X$ via deck transformations. If we endow $M$ with the discrete topology, this is a covering space of $X$ with a distinguished section $0$, the fibers have natural $R$-module structures, and the path lifting isomorphisms $M_x\to M_y$ corresponding to paths $\alpha$ in $X$ running from $x$ to $y$ coincide with the isomorphisms $\tau_{[\alpha]}$ in Definition~\ref{defi:local-system}.

Let now $G$ be any group and assume we are given a faithful representation of $G$ into the automorphisms of some $R$-module $M$, i.e. an injective group homomorphism $G\to \Aut(M)$. We identify $G$ with the orbit of its image under conjugation inside $\Aut(M)$. A \emph{$G$-local system of coefficients with fiber $M$} is a local system with fiber $M$ whose holonomy has image contained in $G$. Equivalently, we require that the holonomy is an element 
$$
\rho\in\Hom(\pi_1(X),G)/G,
$$
where $G$ acts on itself by conjugation. The holonomy representation determines uniquely the $G$-local system up to isomorphism. 

The isomorphism class of a local system is sometimes also called \emph{gauge equivalence class}. The space $\Hom(\pi_1(X),G)/G$ is sometimes called \emph{representation variety of the group $G$}. The correspondence between isomorphism classes of $G$-local systems of coefficients and elements of the representation variety of $G$ is reminiscent of the construction of vector bundles associated to principal bundles through linear representations.

\begin{remark}[On terminology] In the case $R=\Z$ Steenrod~\cite{Steenrod} refers to a local system as being a \emph{bundle of abelian groups}. One also encounters in the literature the notion of a \emph{bundle of groups}, meaning the analogous structure for which the fibers $M_x$ are not $R$-modules, but simply groups. The collection of fundamental groups $\{\pi_1(X;x)\, : \, x\in X\}$ is the prototypical example of bundle of groups. Labourie~\cite[Definition~3.4.7]{Labourie} uses the term \emph{$G$-local system} to  designate a $G$-principal bundle whose transition functions are locally constant, a notion that is very close to that of a bundle of groups with fiber $G$.
\end{remark}

Let now $R=\R,\C$ and $L$ be a finite dimensional vector space over $R$. In this case we can reinterpret local systems of coefficients with fiber modeled on $L$ as vector bundles endowed with flat connections. More precisely, there is a one-to-one bijective correspondence between isomorphism classes of local systems with fiber modeled on $L$ and gauge equivalence classes of pairs consisting of a vector bundle of rank $\dim\, L$ and of a flat connection on this bundle, where two such pairs are gauge equivalent if the vector bundles are isomorphic and if, seen through this isomorphism, the two connections are gauge equivalent. 
Given a local system with holonomy representation $\rho'_{x_0}:\pi_1(X;x_0)\to \Aut(L)$, one associates to it the vector bundle $\tilde X\times_{\pi_1(X;x_0)} L$ with flat connection induced by the trivial connection on $\tilde X\times L$. Then the parallel transport isomorphisms along paths $\alpha$ in $X$ correspond to the isomorphisms $\tau_{[\alpha]}$ in Definition~\ref{defi:local-system}.

Similarly, given a faithful linear representation of a Lie group $G$ into $\mathrm{GL}(L)$, there is a notion of $G$-bundle with flat connection of rank $\dim\, L$~\cite[Definitions~3.2.1 and~3.3.18]{Labourie}. There is a one-to-one bijective correspondence between gauge equivalence classes of $G$-bundles with flat connection and the elements of the representation variety $\Hom(\pi_1(X),G)/G$~\cite[Proposition~3.3.22]{Labourie}. See also~\cite[\S3.5]{Labourie} for further details on the correspondence between points on the representation variety, gauge equivalence classes of $G$-bundles with flat connection, and isomorphism classes of $G$-local systems.

There are natural operations of pull-back, tensor product, and direct sum involving $G$-local systems of coefficients. These can be understood in each of the previous models, and in particular from the bundle point of view, in which case these operations are the usual ones on bundles. 

Let $G$ be an abelian group. In this case the representation variety is simply $\Hom(\pi_1(X),G)$ and has the structure of an abelian group induced by that of $G$. The Hurewicz isomorphism and the universal coefficient theorem imply that
$$
\operatorname{Hom}(\pi_1(X),G)\simeq \operatorname{Hom}(H_1(M;\mathbb{Z}),G)\simeq H^1(M;G).
$$

If $G=\Z/2$ we can represent it as $\{\pm 1\}\subset O(1)$, and we recover the familiar fact that isomorphism classes of real line bundles are in one-to-one bijective correspondence with elements of $H^1(M;\Z/2)$. The information on the connection is redundant here since the structure group of any real line bundle can be reduced to $O(1)$, and any $O(1)$-connection is automatically flat.

Similarly, if $G=\Z/k$ for some $k\ge 2$, we can represent it inside $U(1)$ as the group of $k$-th roots of unity. We recover the fact that isomorphism classes of complex line bundles whose structure group can be reduced to $\Z/k$ are in one-to-one bijective correspondence with elements of $H^1(M;\Z/k)$. Again, the information on the connection is redundant since any connection with holonomy in a discrete group is necessarily flat.

\subsubsection{Homology and cohomology with local coefficients}

Given a $G$-local system of coefficients $\cL=\{M_x\, : \, x\in X\}$ with fiber $M$ on a path-connected topological space $X$ which has a universal cover $\tilde X$, one defines \emph{twisted homology and cohomology groups} $H_\cdot(X;\cL)$ and $H^\cdot(X;\cL)$ as follows~\cite[\S3.H]{Hatcher}. We choose a representative $\rho\in\Hom(\pi_1(X),G)$ for the holonomy, which exhibits $M$ as a left $\Z[\pi_1(X)]$-module. On the other hand $C_\cdot(\tilde X)$ has a left $\Z[\pi_1(X)]$-module structure induced by deck transformations, and also a right $\Z[\pi_1(X)]$-module structure via $ag=g^{-1}a$. We define \emph{homology groups of $X$ with local coefficients in $\cL$} as 
$$
H_\cdot(X;\cL)=H_\cdot(C_\cdot(\tilde X)\otimes _{\Z[\pi_1(X)]} M). 
$$
Similarly, we define \emph{cohomology groups of $X$ with local coefficients in $\cL$} as 
$$
H^\cdot(X;\cL)=H^\cdot(\Hom_{\Z[\pi_1(X)]}(C_\cdot(\tilde X),M)).
$$

An alternative definition is the following, see~\cite[\S3.H]{Hatcher} and also~\cite[\S5.9]{McCleary}. The chain complex $C_\cdot(\tilde X)\otimes _{\Z[\pi_1(X)]}M$ can be identified with the chain complex $C_\cdot(X;\cL)$ whose elements are formal linear combinations $\sum n_i\sigma_i$, where $\sigma_i$ is a singular simplex in $X$ and $n_i$ is a horizontal 
lift of $\sigma_i$ to $\cL$. We also write such a formal linear combination as $\sum n_i\otimes \sigma_i$ in order to emphasize the relation to the previous point of view. The boundary operator is obtained by restricting each $n_i$ to the faces of the standard simplex. Similarly, the cochain complex $\Hom_{\Z[\pi_1(X)]}(C_\cdot(\tilde X),M)$ can be identified with the cochain complex $C^\cdot(X;\cL)$, whose elements are functions which associate to each singular simplex in $X$ a lift of that simplex to $\cL$.

Given a map $f:X\to Y$ and a local system $\cL$ on $Y$, we have natural maps 
\begin{equation}\label{eq:functoriality}
f_*:H_\cdot(X;f^*\cL)\to H_\cdot(Y;\cL)
\end{equation}
and
$$
f^*:H^\cdot(Y;\cL)\to H^\cdot(X;f^*\cL). 
$$

\subsubsection{Products} \label{sec:products}

Cup products and cap products in twisted homology or cohomology are only defined provided one works in the category of algebras over a ring $R$, instead of the category of $R$-modules as above. Given an $R$-algebra $K$ and a faithful representation of a group $G$ into the group of $R$-algebra automorphisms of $K$, we have an obvious notion of $G$-local system with fiber $K$. The discussion in the previous paragraph holds verbatim in order to show that such local systems are again classified by 
$$
\Hom(\pi_1(X),G)/G. 
$$
Note that the only $R$-algebra automorphism of $R$ is the identity, so that there are no nontrivial local systems of $R$-algebras with fiber $R$. We refer to~\cite[\S3.H]{Hatcher} for a discussion of cup- and cap products, as well as of Poincar\'e duality.

In this paper we shall use the Pontryagin product on twisted homology of an $H$-space, with the based loop space in mind. The striking difference with respect to the cup-product is that the Pontryagin product is defined for rank $1$ local systems of $R$-modules, i.e. local systems for which the model fiber is a ring $R$ and whose holonomy consists of $R$-module automorphisms of $R$, as we shall now explain.

Let $R$ be a commutative ring and denote $R^\times$ the group of invertible elements in $R$. Any $R$-module automorphism of $R$ is given by multiplication with an element of $R^\times$. Indeed, if $\varphi\in\Aut_{R-\mathrm{mod}}(R)$ then $\varphi(1)\in R^\times$ and $\varphi(x)=\varphi(1)x$. Also, given a representation $\rho:\pi_1(X)\to \Aut_{\Z-\mathrm{mod}}(R)$, the condition that $\rho$ takes values in $\Aut_{R-\mathrm{mod}}(R)$ is equivalent to the identity
$$
\rho(g_1)(x_1) \rho(g_2)(x_2)=\rho(g_1g_2)(x_1x_2),\qquad \forall g_1, g_2, x_1, x_2.
$$
The examples that we shall be mainly using are the following: if $R=\Z$ then $R^\times=\{\pm 1\}$, if $R=\R$ then $R^\times=\R^*$, and if $R=\C$ then $R^\times=\C^*$. The group $\Z/2$ is isomorphic to $\Z^\times$ and can be realized as a subgroup of $\R^\times$ and $\C^\times$. Any finite cyclic group $\Z/k$, $k\ge 2$ can be realized as a subgroup of $\C^\times$ by viewing it as the subgroup $\{1,\zeta,\zeta^2,\dots,\zeta^{k-1}\}$ generated by a primitive $k$-th root of unity. 

For the next proposition, recall that an \emph{$H$-space} is a topological space endowed with a continuous map $m:X\times X\to X$ called \emph{multiplication}, and with an element $e\in X$, called \emph{homotopy unit}, such that $m(e,\cdot):X\to X$ and $m(\cdot,e):X\to X$ are homotopic to the identity through maps that preserve the basepoint $e$, so that in particular $m(e,e)=e$. We also require the multiplication to be associative up to homotopy.

\begin{proposition} \label{prop:Hspace} Let $R$ be a ring and $X$ be a path-connected $H$-space. For any rank one $G$-local system of $R$-modules $\cL$ on $X$, the multiplication $m:X\times X\to X$ induces a unital ring structure on $H_\cdot(X;\cL)$. The unit is represented by the class of a point. 
\end{proposition}

The ring structure is defined at chain level as the composition
$$
\xymatrix
@C=30pt
{
C_\cdot(X;\cL)\otimes_R C_\cdot(X;\cL)\ar[r]^-B & C_\cdot(X\times X;pr_1^*\cL\otimes_R pr_2^*\cL)\ar[r]^-{m_*} & C_\cdot(X;\cL).
}
$$
Here $B$ is the Eilenberg-MacLane shuffle map with local coefficients and $m_*$ is induced by the multiplication, as explained below.

That $B$ acts as indicated is a consequence of its definition. Indeed, given spaces $X$ and $Y$ with local systems of $R$-modules $\cL_X$ and $\cL_Y$, the \emph{Eilenberg-MacLane shuffle map}  
$$
B:C_\cdot(X;\cL_X)\otimes_R C_\cdot(Y;\cL_Y)\to C_\cdot(X\times Y;pr_1^*\cL_X\otimes_R pr_2^*\cL_Y)
$$
is defined as follows (see~\cite[p.~64]{EMacL53b} and~\cite[p.~268]{Greenberg-Harper} for the untwisted version). Given $\sigma:\Delta^p\to X$ and $\tau:\Delta^q\to Y$, and given lifts $\alpha:\Delta^p\to \cL_X$ and $\beta:\Delta^q\to \cL_Y$ of $\sigma$ and $\tau$ to $\cL_X$, respectively $\cL_Y$, we define 
$$
B(\alpha\otimes\sigma\bigotimes\beta\otimes\tau)=\sum \pm(\alpha\circ D_I\otimes\beta\circ D_J)\otimes (\sigma\circ D_I,\tau\circ D_J),
$$
where the sum ranges over $(p,q)$-shuffles $(I,J)$, the sign is given by the signature of the corresponding permutation, $D_I$ and $D_J$ are the corresponding projections $\Delta^{p+q}\to\Delta^p$ and $\Delta^{p+q}\to\Delta^q$. One checks readily that $\alpha\circ D_I\otimes \beta\circ D_J$ is a horizontal lift of $(\sigma\circ D_I,\tau\circ D_J)$ to $pr_1^*\cL_X\otimes pr_2^*\cL_Y$.

That $m_*$ acts as indicated is a consequence of~\eqref{eq:functoriality} and of the following Lemma. 

\begin{lemma} Let $X$ be a path-connected $H$-space with multiplication $m:X\times X\to X$ and homotopy unit $e$. For any rank one local system of $R$-modules $\cL$ on $X$, we have a canonical isomorphism  
$$
pr_1^*\cL\otimes_R pr_2^*\cL\cong m^*\cL
$$
which coincides with the canonical isomorphism $R\otimes_R R\cong R$ on the fiber at $(e,e)$. 
\end{lemma}

\begin{proof} Denote $\varphi:R\otimes_R R\stackrel\simeq\longrightarrow R$, $a\otimes b\mapsto ab$ the canonical isomorphism between the model fibers at $(e,e)$. Then $\varphi^{-1}$ acts by $a\mapsto a\otimes 1=1\otimes a$ and we have an induced group isomorphism 
$$
I_\varphi:\Aut_{R-\mathrm{mod}}(R\otimes_R R)\to \Aut_{R-\mathrm{mod}}(R), \qquad f\mapsto \varphi f \varphi^{-1}.
$$
We need to show that the following diagram commutes
$$
\xymatrix
@C=60pt
{\pi_1(X\times X;(e,e)) \ar[r]^-{\rho^{m^*\cL}}Ê\ar[dr]_-{\rho^{pr_1^*\cL\otimes pr_2^*\cL}\quad }Ê& \Aut_{R-\mathrm{mod}}(R) \\
& \Aut_{R-\mathrm{mod}}(R\otimes_R R) \ar[u]_-{I_\varphi},
}
$$
where $\rho^{m^*\cL}$ and $\rho^{pr_1^*\cL\otimes pr_2^*\cL}$ are the holonomy representations for $m^*\cL$ and $pr_1^*\cL\otimes_R pr_2^*\cL$ respectively. 

Let $\gamma=(\gamma_1,\gamma_2)$ be a loop in $X\times X$ based at $(e,e)$. It is a general fact that $m\circ \gamma$ is homotopic to the concatenation $\gamma_1\cdot\gamma_2$, as well as to the concatenation $\gamma_2\cdot \gamma_1$ (this is the general reason why $H$-spaces have abelian fundamental group). Thus 
$$
\rho^{m^*\cL}(\gamma)=\rho^\cL(m\circ\gamma)=\rho^\cL(\gamma_2\cdot\gamma_1)=\rho^\cL(\gamma_1)\circ \rho^\cL(\gamma_2),
$$
so that
$$
\rho^{m^*\cL}(\gamma)a=\rho^\cL(\gamma_1)(1) \rho^\cL(\gamma_2)(1) a,\qquad a\in R.
$$
On the other hand 
\begin{eqnarray*}
I_\varphi\rho^{pr_1^*\cL\otimes pr_2^*\cL}(\gamma) & = & \varphi \rho^{pr_1^*\cL\otimes pr_2^*\cL}(\gamma)\varphi^{-1}Ê\\
& = & \varphi (\rho^{pr_1^*\cL}(\gamma)\otimes\mathrm{Id})\circ (\mathrm{Id}\otimes\rho^{pr_2^*\cL}(\gamma))\varphi^{-1}Ê\\
& = & \varphi (\rho^\cL(\gamma_1)\otimes\mathrm{Id})\circ (\mathrm{Id}\otimes\rho^\cL(\gamma_2))\varphi^{-1}
\end{eqnarray*}
acts by 
\begin{equation*}
\begin{split}
a\mapsto 1\otimes a \mapsto 1\otimes \rho^\cL(\gamma_2)(a) & \mapsto \rho^\cL(\gamma_1)(1)\otimes  \rho^\cL(\gamma_2)(a) \\
& \mapsto \rho^\cL(\gamma_1)(1)  \rho^\cL(\gamma_2)(a) = \rho^\cL(\gamma_1)(1) \rho^\cL(\gamma_2)(1) a
\end{split}
\end{equation*}
for any $a\in R$. 
\end{proof}

\begin{proof}[Proof of Proposition~\ref{prop:Hspace}]
Associativity follows from the associativity up to homotopy of the multiplication $m$, and from the associativity of the multiplication in the ring $R$. The property of being a homotopy unit implies that $1\otimes e$ is a cycle whose homology class is a unit for the previous ring structure. Since any point in a path-connected $H$-space is a homotopy unit, it follows that the same class can also be represented as $1\otimes x$ for any point $x\in X$. 
\end{proof}

The following is a particular case of Proposition~\ref{prop:Hspace}.

\begin{proposition} \label{prop:Pontryagin}
Let $M$ be a path-connected space 
with basepoint and denote $\Omega_0M$ the component of the constant loop inside the based loop space. For any ring $R$ and any rank one $G$-local system of $R$-modules $\cL$ on $\Omega_0M$, the concatenation of loops induces a unital ring structure on $H_\cdot(\Omega_0M;\cL)$. The unit is represented by the class of the constant loop.
\end{proposition}

\begin{proof}
This follows from Proposition~\ref{prop:Hspace}, since $\Omega_0M$ is an $H$-space with multiplication given by the concatenation of loops. The homotopy unit is the constant loop at the basepoint. 
\end{proof}

\noindent {\sc Applications.}

\emph{(i) Vanishing of twisted homology of the based loop space.}

\begin{proposition} \label{prop:vanishing_H-space}
Let $X$ be a path-connected $H$-space. Let $R$ be a ring and $G\subset R^\times$ a subgroup such that 
$$
1-g\in R^\times
$$
for any $g\in G$ with $g\neq 1$. Then any nontrivial $G$-local system $\cL$ of rank one $R$-modules on $X$ has the property that $H_\cdot(X;\cL)=0$. 

This holds in particular for any nontrivial rank one local system if $R$ is a field.
\end{proposition}

\begin{proof} Since $H_\cdot(X;\cL)$ is a ring with unit, it is enough to show that the unit $[1\otimes e]\in H_0(X;\cL)$ vanishes, where $e\in X$ is the homotopy unit. To prove this, we identify the fiber of $\cL$ at $e$ with $R$ and choose a loop $\gamma:[0,1]\to X$, $\gamma(0)=\gamma(1)=e$, along which the holonomy is nontrivial, and as such given by multiplication with some $g\in G$, $g\neq 1$. Let $\tilde x:[0,1]\to\cL$ be the unique lift of $\gamma$ such that $\tilde x(1)=1$. Then 
$$
\p(\tilde x\otimes \gamma)=(1-g^{-1})\otimes e.
$$ 
By assumption $1-g^{-1}\in R^\times$ and thus 
$$
\p \left( (1-g^{-1})^{-1}\tilde x\otimes \gamma\right) = 1\otimes e.
$$
\end{proof}

\begin{corollary} \label{cor:vanishing_based_loops}
Let $M$ be a path-connected space
with basepoint and let $\Omega_0M$ be the component of the constant loops in the based loop space. Let $R$ be a ring and $G\subset R^\times$ a subgroup such that 
$$
1-g\in R^\times
$$
for any $g\in G$ with $g\neq 1$. Then any nontrivial $G$-local system $\cL$ of rank one $R$-modules on $\Omega_0M$ has the property that $H_\cdot(\Omega_0M;\cL)=0$. 

This holds in particular for any nontrivial rank one local system if $R$ is a field.
\end{corollary}

\begin{proof} This follows from the previous Proposition using the fact that $\Omega_0M$ is a path-connected $H$-space.
\end{proof}

In the next section we find necessary and sufficient conditions on a manifold $M$ for such nontrivial local systems to exist.

\emph{(ii) $H$-spaces.}

Our next result is a homological obstruction to the existence of $H$-space structures. 
It is obtained by negating Proposition~\ref{prop:vanishing_H-space}.

\begin{proposition} \label{prop:RPn}
Let $X$ be a path-connected topological space. Assume that there exists a ring $R$, a subgroup $G\subset R^\times$ such that $1-g\in R^\times$ for any $g\in G$, $g\neq 1$, and a $G$-local system $\cL$ of rank one $R$-modules which is nontrivial and such that $H_\cdot(X;\cL)\neq 0$. Then $X$ does not carry the structure of an $H$-space. 
\hfill{$\square$}
\end{proposition}

\begin{example}[Even dimensional real projective spaces $\R P^{2k}$, $k\ge 1$ do not carry $H$-space structures] We consider on $\R P^n$ the unique nontrivial rank one local system $\cL$ with fiber $\R$ and holonomy $-\mathrm{Id}$ along the generator of $\pi_1(\R P^n)$. To compute $H_\cdot(\R P^n;\cL)$ we use a cellular decomposition on $\R P^n$ with one cell in each dimension between $0$ and $n$. The resulting chain complex is
$$
\xymatrix
@C=40pt
{
\R & \ar[l]_{\times 2} \R & \ar[l]_0 \R & \ar[l]_{\times 2} \dots & \ar[l] \R,
}
$$
where the last copy of $\R$ sits in degree $n$. This complex is acyclic if $n$ is odd, and its homology has rank one and is supported in degree $n$ if $n$ is even. In view of Proposition~\ref{prop:RPn}, we conclude that even dimensional real projective spaces are not $H$-spaces. The classical proof of this result uses a structure theorem for Hopf algebras~\cite{Hopf_Uber_die_Topologie_der_Gruppenmanigfaltigkeiten} and \cite[Section 3.C]{Hatcher}.

The obstruction vanishes for odd-dimensional real projective spaces. As a matter of fact these do carry sometimes $H$-space structures: $\R P^1$ is diffeomorphic to $SO(2)$, and $\R P^3$ is diffeomorphic to $SO(3)$. Finally, $\R P^7$ carries the structure of an $H$-space since the octonions give $S^7$ the structure of an $H$-space and $\pm1$ belong to its center. Thus $\R P^7$ inherits its $H$-space structure from $S^7$, see \cite[Example 3.C.3]{Hatcher}. That $\R P^{2k+1}$, $k\ge 4$ are not $H$-spaces is a consequence of a classical theorem of Adams, which states that the only spheres that are $H$-spaces are $S^0$, $S^1$, $S^3$, and $S^7$. Indeed, should $\R P^{2k+1}$, $k\ge 4$ have an $H$-space structure, the same would be true for its universal cover $S^{2k+1}$.  

In a similar vein, note that $\R P^\infty=K(\Z/2,1)$ has an $H$-space structure induced by the ring structure of $\Z/2$, and indeed $H_\cdot(\R P^\infty;\cL)=0$ for the nontrivial local system described above. 
\end{example}

\subsubsection{Leray-Serre spectral sequence for homology with local coefficients}

The following is a variation on the classical theme of the Leray-Serre spectral sequence. 

\begin{proposition} \label{prop:LS} Let $F\stackrel i \hookrightarrow E\stackrel \pi\longrightarrow B$ be a fibration whose base is a CW-complex, and let $\cL$ be a local system on $E$. There is a spectral sequence $(E^r_{p,q},d^r)$, $r\ge 2$ converging to $H_\cdot(E;\cL)$ and with second page 
$$
E^2_{p,q}\simeq H_p(B;\cH_q(F;i^*\cL)),\qquad p,q\ge 0.
$$
Here $\cH_\cdot(F;i^*\cL)$ denotes the local system on $B$ defined by the homology of the fiber. 
\end{proposition}

\begin{proof}
Let us describe the local system $\cH_\cdot(F;i^*\cL)$. We assume for simplicity that the fibration is locally trivial and that $F$, $E$, $B$ are smooth manifolds, so that we can pick a connection on $E$ that allows to lift paths on the base (in the general case, one can use a lifting function, see for example~\cite[\S4.3]{McCleary}). Given a path $\gamma:[0,1]\to B$, $\gamma(0)=x$, $\gamma(1)=y$ we denote $\Phi_\gamma:F_{x}\stackrel\simeq\longrightarrow F_{y}$ the monodromy diffeomorphism defined by lifting $\gamma$ with arbitrary starting point in $F_{x}$. Given $e\in F_x$ we denote $\tilde \gamma_e$ the unique horizontal lift of $\gamma$ with starting point $e$.

The monodromy induces an isomorphism 
$$
\Phi_{\gamma *}:H_\cdot(F_x;\Phi_\gamma^*\cL|_{F_y})\stackrel\simeq\longrightarrow H_\cdot(F_y;\cL|_{F_y}).
$$
The point now is that we have an isomorphism 
$$
\cT_{\Phi_\gamma}: \cL|_{F_x}\stackrel\simeq\longrightarrow \cL|_{F_y}
$$
given by the holonomy of $\cL$ with respect to all the lifts of $\gamma$ starting at arbitrary points in $F_x$. This isomorphism can be equivalently written as an isomorphism 
$$
\cT_{\Phi_\gamma}':\cL|_{F_x}\stackrel\simeq\longrightarrow \Phi_\gamma^*\cL|_{F_y},
$$
which together with $\Phi_\gamma$ induces an isomorphism 
$$
\Phi_{\gamma *}':H_\cdot(F_x;\cL|_{F_x})\stackrel\simeq\longrightarrow H_\cdot(F_y;\cL|_{F_y}).
$$
It is easy to verify that these isomorphisms depend only on the homotopy class of $\gamma$ with fixed endpoints, and that they are compatible with concatenation, so that they define a local system on $B$ whose fiber at any point $x$ is $H_\cdot(F_x;\cL|_{F_x})$. This local system is denoted $\cH_\cdot(F;i^*\cL)$.

The spectral sequence arises from the filtration on $C_\cdot(E;\cL)$ which is determined by the skeleton filtration on $B$. The computation of the second page is \emph{mutatis mutandis} the same as in the case of constant coefficients, see for example~\cite[\S5.3]{McCleary} for a proof in the setup of singular homology.  
\end{proof}

\subsection{Local systems on loop spaces} \label{sec:loc-sys-loop-space}

Let $M$ be a connected manifold, denote $\pi:\tilde M\to M$ its universal cover, and choose a basepoint in $M$. Denote $\Omega M$ the based loop space and $\Omega_0M$ the component of contractible loops. Denote $L M$ the free loop space and $L_0M$ the component of the contractible loops. Let $\ell\ge 2$ be a prime number.

\emph{Local systems on $\Omega_0M$}. The $\Z/\ell$-local systems on $\Omega_0M$ form an abelian group which is isomorphic to any of the following:
\begin{equation} \label{eq:Hompi1Omega0M}
\begin{aligned}
\Hom(\pi_1(\Omega_0M),\Z/\ell) & \cong \Hom(\pi_2(M),\Z/\ell) \\
& \cong \Hom(\pi_2(\tilde M),\Z/\ell)\cong H^2(\tilde M,\Z/\ell). 
\end{aligned}
\end{equation}
The last isomorphism follows from the universal coefficient theorem in view of $H_1(\tilde M)=0$. Note that $\pi_1(M)$ acts on $\tilde M$ by deck transformations and on $\pi_2(M)=\pi_1(\Omega_0M)$ by conjugation, and the above isomorphisms are equivariant with respect to the induced actions. In particular, there exists a non-trivial $\Z/\ell$-local system on $\Omega_0M$ if and only if $H^2(\tilde M;\Z/\ell)\neq 0$. Moreover, this local system can be chosen to satisfy the condition in the statement of Corollary~\ref{cor:vanishing_based_loops} by taking the fiber to be the field $\C$.

\emph{Local systems on $L_0M$}. The $\Z/\ell$-local systems on $L_0M$ form an abelian group which is isomorphic to $\Hom(\pi_1(L_0M),\Z/\ell)$. To understand this group we start from the split short exact sequence
$$
\xymatrix{
0\ar[r] & \pi_1(\Omega_0M)\ar[r]^i & \pi_1(L_0M) \ar[r]^-\pi & \pi_1(M) \ar@/^/[l]^-s \ar[r] & 1
}
$$
induced by the fibration $\Omega_0M\to L_0M\stackrel \pi\to M$, in which $s$ is induced by the inclusion of constant loops and $\pi$ is the evaluation at the starting point of a loop. In such a situation we have an action of $\pi_1(M)$ on $\pi_1(\Omega_0M)$ by conjugation given by $b(a)=i^{-1}(s(b)i(a)s(b)^{-1})$ for $a\in \pi_1(\Omega_0M)$ and $b\in \pi_1(M)$. This action agrees with the action described above for the based loop space. In particular we have an isomorphism 
$$
\pi_1(L_0M)\cong \pi_1(\Omega_0M)\rtimes \pi_1(M)
$$
which expresses $\pi_1(L_0M)$ as the semi-direct product of $\pi_1(\Omega_0M)$ and $\pi_1(M)$. Recall that the underlying set for the semi-direct product is $\pi_1(\Omega_0M)\times\pi_1(M)$ and the group structure is $(a,b)\cdot(a_1,b_1)=(a+b(a_1),bb_1)$. 

As a consequence, we have an isomorphism
\begin{equation}\label{eq:Hompi1L0M}
\Hom(\pi_1(L_0M),\Z/\ell)\cong \Hom_{inv}(\pi_1(\Omega_0M),\Z/\ell)\times \Hom(\pi_1(M),\Z/\ell),
\end{equation}
where $\Hom_{inv}(\pi_1(\Omega_0M),\Z/\ell)\subseteq \Hom(\pi_1(\Omega_0M),\Z/\ell)$ is the subgroup of $\pi_1(M)$-invariant elements. The above isomorphism associates to any element $\tilde h$ in the group $\Hom(\pi_1(L_0M),\Z/\ell)$ the pair $(\tilde h(\cdot,1),\tilde h(0,\cdot))$, and to a pair $(h,\rho)$ the element $\tilde h:(a,b)\mapsto h(a)+\rho(b)$.

Since the isomorphisms in~\eqref{eq:Hompi1Omega0M} are $\pi_1(M)$-equivariant we have that 
$$
\Hom_{inv}(\pi_1(\Omega_0M),\Z/\ell)\cong \Hom_{inv}(\pi_2(M),\Z/\ell)\cong H^2_{inv}(\tilde M;\Z/\ell).
$$
To get a better grasp on the group $H^2_{inv}(\tilde M;\Z/\ell)$, note that we have inclusions 
$$
\pi^*H^2(M;\Z/\ell)\subseteq H^2_{inv}(\tilde M;\Z/\ell)\subseteq H^2(\tilde M;\Z/\ell).
$$
Either inclusion can be strict, as the following two examples show. 

\noindent \emph{Examples.} (i) Take $\ell=2$, $M=\R P^2$, $\tilde M=S^2$. Then $0=\pi^*H^2(\R P^2;\Z/2)\subset H^2_{inv}(S^2;\Z/2)=H^2(S^2;\Z/2)=\Z/2$ since $\pi$ has degree $2$ and thus $\pi^*$ vanishes with $\Z/2$-coefficients. 

(ii) Take $M$ to be the $4$-torus blown-up at a point, so that $\tilde M$ is $\R^4$ blown-up at all the points in the integral lattice. The group $\pi_1(M)\cong \Z^4$ acts on $\tilde M$ by translations. Thus $\pi^*H^2(M;\Z/\ell)=H^2_{inv}(\tilde M;\Z/\ell)\cong \Z/\ell\subset H^2(\tilde M;\Z/\ell)=\Hom(\Z^4,\Z/\ell)$.

\emph{Restriction and extension of local systems}. Any local system on $L_0M$ restricts to a local system on $\Omega_0M$. Seen through the isomorphism~\eqref{eq:Hompi1L0M}, restriction corresponds to projection onto the first factor
$$
\Hom(\pi_1(L_0M),\Z/\ell)\to \Hom_{inv}(\pi_1(\Omega_0M),\Z/\ell).
$$
It follows that a necessary and sufficient condition for a local system on $\Omega_0M$ to extend to a local system on $L_0M$ is to be $\pi_1(M)$-invariant. Given such an invariant  local system $\cL$ on $\Omega_0 M$, there is a canonical extension to a local system on $L_0M$ which is trivial on $\pi_1(M)$, denoted $\tilde \cL$.

Similarly, the second projection
$$
\Hom(\pi_1(L_0M),\Z/\ell)\to \Hom(\pi_1(M),\Z/\ell)
$$
corresponds to the restriction of a local system on $L_0M$ to a local system on $M$, where we see the latter as the subspace of constant loops. Any local system on $M$ extends to a local system on $L_0M$ which is constant on $\Omega_0M$ by pull-back via the evaluation map $L_0M\stackrel\pi\to M$. 

\emph{Non-trivial local systems}. We are interested in the existence of non-trivial local systems on $L_0M$ which restrict to trivial local systems on $M$. In other words, we want to derive necessary and sufficient conditions for the nontriviality of the group $H^2_{inv}(\tilde M;\Z/\ell)$. To study this group one views $\tilde M$ as a $G$-space with $G=\pi_1(M)$, so that $H^\cdot(\tilde M)$ is a $G$-module. As such we have 
$$
H^2_{inv}(\tilde M)=H^0(G;H^2(\tilde M)),
$$
so that $H^2_{inv}(\tilde M)$ can be studied using the Cartan-Leray spectral sequence~(\cite{Cartan-Leray}, \cite[Theorem~8$^{\mathrm{bis}}$.9]{McCleary}). The latter can be seen as a particular case of the Leray-Serre spectral sequence for the locally trivial fibration 
$$
\tilde M \hookrightarrow \tilde M\times_G BG \stackrel{pr_2}\longrightarrow BG.
$$
Since $G$ acts freely on $\tilde M$ the total space of this fibration is homotopy equivalent to $M=\tilde M/G$, and thus the spectral sequence converges to $H^\cdot(M)$ with second page given by $E_2^{p,q}=H^p(G;H^q(\tilde M))$. Since $H^1(\tilde M)=0$, the definition of convergence implies the following classical exact sequence in low cohomological degrees (\cite{Hopf42}, \cite[Theorem~8$^{\mathrm{bis}}$.10]{McCleary}). 

\begin{proposition} Let $G=\pi_1(M)$ and let $\tilde M\stackrel\pi\longrightarrow M$ be the universal cover of a manifold $M$. There is an exact sequence (with arbitrary coefficients)
\begin{equation} \label{eq:Cartan-Leray}
0\longrightarrow H^2(G)\longrightarrow H^2(M)\stackrel{\pi^*}\longrightarrow H^2_{inv}(\tilde M)\longrightarrow H^3(G)\longrightarrow H^3(M),
\end{equation}
in which the maps $H^\cdot(G)\to H^\cdot(M)$ are induced by the classifying map $M\to BG$ of the principal $G$-bundle $\tilde M\stackrel\pi\longrightarrow M$.
\hfill{$\square$}
\end{proposition}

The following two statements are straightforward corollaries. 

\begin{proposition} For a fixed choice of coefficients, the group $H^2_{inv}(\tilde M)$ vanishes if and only if the map $H^2(G)\to H^2(M)$ is surjective (an isomorphism) and the map $H^3(G)\to H^3(M)$ is injective.
\hfill{$\square$}
\end{proposition}

\begin{proposition} \label{prop:existence_local_system}
There exists a prime number $\ell$ and a nontrivial $\Z/\ell$-local system on $L_0M$ which restricts to a trivial local system on $M$ if and only if the map $H^2(G;\Z/\ell)\to H^2(M;\Z/\ell)$ is not surjective, or the map $H^3(G;\Z/\ell)\to H^3(M;\Z/\ell)$ is not injective. In addition, this local system can be chosen to satisfy the condition in the statement of Corollary~\ref{cor:vanishing_based_loops}.
\hfill{$\square$}
\end{proposition}

\begin{example} \label{ex:siconnH2}
A general class of manifolds for which there exists a prime $\ell$ such that the map $H^2(G;\Z/\ell)\to H^2(M;\Z/\ell)$ is not surjective is that of simply-connected manifolds with $\pi_2(M)=H_2(M;\Z)\neq 0$. This contains the class of closed simply-connected symplectic manifolds, and in particular the $2$-sphere. 
\end{example}

The previous example is subsumed by the following criterion for non-surjectivity of the map $H^2(G)\to H^2(M)$.

\begin{proposition} \label{prop:pi2}
The map $H^2(G;\Z/\ell)\to H^2(M;\Z/\ell)$ is not surjective for some prime $\ell$ if and only if the Hurewicz map $\pi_2(M)\to H_2(M;\Z)$ is nonzero. 
\end{proposition}

\begin{proof}
This is a consequence of the Hopf exact sequence~(\cite{Hopf42}, \cite[p.~339]{McCleary})
$$
\pi_2(M)\stackrel{h}\longrightarrow H_2(M;\Z)\stackrel{f_*}\longrightarrow H_2(G;\Z)\to 0,
$$
in which the first map $h$ is the Hurewicz homomorphism, and the second map $f_*$ is induced by the classifying map $f:M\to BG$ for the principal $G$-bundle $\tilde M\to M$. The non-vanishing of $h$ is equivalent to the map $f_*$ not being an isomorphism. By the universal coefficient theorem and the fact that $f_*:H_1(M;\Z)\to H_1(G;\Z)$ is an isomorphism, this is equivalent to  $f^*:H^2(G;\Z)\to H^2(M;\Z)$ not being an isomorphism, which is equivalent to $f^*:H^2(G;\Z) \to H^2(M;\Z)$ not being surjective. The conclusion follows. 
\end{proof}

\begin{example} \label{ex:stability} The condition in Proposition~\ref{prop:pi2} is preserved under taking products with one such manifold, as well as under connected sums with one such summand in dimension at least $3$. 
\end{example}

\begin{example} \label{ex:surgery}
One can construct examples such that the map $H^3(G)\to H^3(M)$ is not injective using surgery. More precisely, we construct examples such that $H^3(M)=0$ and $H^3(G)\neq 0$. Indeed, for any finitely presented group $G$ such that $H^3(G)\neq 0$ one can build a closed manifold $M$ of dimension $\ge 6$ with $\pi_1(M)\simeq G$ and $H^3(M)=0$. Given a presentation of $G$ with generators $g_1,\dots,g_n$ and relations $r_1,\dots,r_m$, start with a connected sum of $n$ copies of $S^1\times S^k$ with $k\ge 5$, represent the relations $r_1,\dots,r_m$ by $m$ disjoint embedded loops, then perform surgery in order to kill those. The fundamental group of the resulting manifold is $G$. We now show that $H^3(M)=0$. Indeed, $M$ is a CW-complex without 3-cells. Therefore, there is no torsion in $H_2(M)$ and $H_3(M)=0$. It follows from \cite[Theorem 3.2]{Hatcher} that $H^3(M)\cong\mathrm{Hom}(H_3(M),\Z)\oplus\mathrm{Ext}(H_2(M),\Z)=\mathrm{Ext}(H_2(M),\Z)$. Since $H_2(M)$ is finitely generated and torsion-free it is actually free and therefore $\mathrm{Ext}(H_2(M),\Z)=0$, see \cite{Hatcher}.
\end{example}

\noindent {\sc Application.}

\emph{Vanishing of twisted homology of the free loop space.}

\begin{theorem}\label{thm:vanishing-homology} Let $M$ be a manifold such that the image of the Hurewicz map 
$$
\pi_2(M)\to H_2(M;\Z)
$$
is nonzero. For any rank $1$ local system $\cC$ on $M$ there exists a local system of coefficients $\cL$ on $L_0M$ whose restriction to the space of constant loops $M\subset L_0M$ is isomorphic to $\cC$, and such that 
$$
H_\cdot(L_0M;\cL)=0.
$$ 
The same conclusion holds more generally if condition (C) in the Introduction holds. 
\end{theorem}

\begin{proof} By Proposition~\ref{prop:pi2} and Proposition~\ref{prop:existence_local_system} there exists a prime number $\ell$ and a nontrivial $\Z/\ell$-local system $\tilde \cL$ which satisfies the condition in Corollary~\ref{cor:vanishing_based_loops}
and which restricts to the trivial local system on $M$. By that same corollary we have $H_\cdot(\Omega_0M;\tilde \cL|_{\Omega_0M})=0$. 

Let $\pi:L_0M\to M$ be the evaluation map. We claim that $\cL=\tilde \cL\otimes \pi^*\cC$ satisfies the conclusion of the Theorem. Indeed, since $\cL|_{\Omega_0M}=\tilde \cL|_{\Omega_0 M}$ we infer that the Leray-Serre spectral sequence in Proposition~\ref{prop:LS} for the fibration $\Omega_0M\to L_0M\stackrel\pi\to M$ has trivial second page, and its target $H_\cdot(L_0M;\cL)$ must necessarily be zero.   
\end{proof}

\subsection{More on the Pontryagin product with twisted coefficients} We discuss in this section the Pontryagin product in twisted homology for $\Omega M$ (as opposed to $\Omega_0M$ as in the previous paragraphs). The material in this section in not used elsewhere in the paper, but we include it for the sake of completeness.

Fix a basepoint on the connected manifold $M$. In view of the equality $\pi_0(\Omega M)=\pi_1(M)$, we label the connected components of $\Omega M$ as $\Omega_gM$, $g\in \pi_1(M)$. Given any element $g\in\pi_1(M)$ we have homotopy-equivalences which are well-defined up to homotopy 
$$
L_g:\Omega_hM\to \Omega_{gh}M,\qquad R_g:\Omega_hM\to \Omega_{hg}M,
$$
given by concatenation on the left, respectively on the right with a representative of $g$. In an equivalent formulation these are the left-, respectively right translation by (a representative of) $g$ in the $H$-space $\Omega M$. 

Given a local system $\cL$ on $\Omega_0M$, we obtain two local systems 
$L^*\cL$ and $R^*\cL$ on $\Omega M$ which are well-defined up to isomorphism, given by 
$$
(L^*\cL)|_{\Omega_gM}=L_g^*\cL, \qquad (R^*\cL)|_{\Omega_gM}=R_g^*\cL,\qquad g\in \pi_1M.
$$

If $\cL$ is a $G$-local system on $\Omega_0M$ which is invariant under conjugation, i.e. the corresponding holonomy representation belongs to $\Hom_{inv}(\pi_1(\Omega_0M),G)/G$ where $\pi_1(M)$ acts as usual at the source by conjugation, then $L_g^*\cL=R_g^*\cL$ for all $g\in \pi_1(M)$. Indeed, denoting $conj_g:\Omega M\to \Omega M$, $x\mapsto g xg^{-1}$, we have $L_g=R_g\circ conj_g$, whereas invariance under conjugation translates into $conj_g^*\cL=\cL$. We denote the corresponding extension of $\cL$ to $\Omega M$ by 
$$
\tilde\cL=L^*\cL=R^*\cL.
$$

\begin{proposition} \label{prop:Pontryagin-OmegaM} Let $R$ be a ring. For any rank one $G$-local system of $R$-modules $\cL$ on $\Omega_0M$ which is invariant under conjugation, the Pontryagin product defines a unital ring structure on $H_\cdot(\Omega M;\tilde \cL)$. The unit is represented by the class of the constant loop.
\end{proposition}

\begin{proof}
The Pontryagin product acts naturally as 
$$
H_\cdot(\Omega_gM;L_{g^{-1}}^*\cL)\otimes H_\cdot(\Omega_hM;R_{h^{-1}}^*\cL)\to H_\cdot(\Omega_{gh}M;(L_{g^{-1}}\circ R_{h^{-1}})^*\cL).
$$
This is a consequence of the discussion leading to Proposition~\ref{prop:Pontryagin} and of the diagram  
$$
\xymatrix
@C=60pt
{
\Omega_gM\times\Omega_hM\ar[r]^-c & \Omega_{gh}M\\
\Omega_0M\times\Omega_0M \ar@<1ex>[u]^{L_g} \ar@<-1ex>[u]_{R_h} \ar[r]^-c & \Omega_0M, \ar[u]_{L_g\circ R_h}
}
$$
which commutes up to homotopy and which exhibits the previous operation as being induced by the corresponding operation on $H_\cdot(\Omega_0M;\cL)$. The conclusion follows readily in case the local system $\cL$ is conjugation invariant. 
\end{proof}
 
\noindent {\bf Remark.} We find it remarkable that the class of local systems for which the Pontryagin product is defined on twisted homology of $\Omega M$ is the same as that of local systems on $\Omega_0M$ which extend to $L_0M$, i.e. local systems which are invariant under conjugation.

%
%
%


\section{Vanishing of symplectic homology with local coefficients and applications}

We present in this section two applications of our vanishing result for the twisted homology of free loop spaces. The first concerns the existence of nonconstant pseudoholomorphic planes in certain cotangent bundles. The second is a finiteness result for the $\pi_1$-sensitive Hofer-Zehnder capacity. 

Both our applications rely on the following vanishing theorem, which is a straightforward consequence of the results of~\cite{Abouzaid-cotangent,AS-corrigendum} together with the discussion in~\S\ref{sec:local_systems}, 
see also~\cite{Viterbo-cotangent, AS, SW, Ritter} for previous versions of the theorem valid for slightly less general coefficients. 
For the statement recall that the \emph{orientation local system} $o_M$ of a manifold $M$ is the $\Z/2$-local system whose holonomy along a loop $\gamma$ equals $\mathrm{Id}$ if the orientation is preserved by parallel transport along $\gamma$, and $-\mathrm{Id}$ if the orientation is reversed by parallel transport along $\gamma$. Given a local system $\cL$ on $L_0T^*M$, the space of contractible loops on $T^*M$, we denote $SH_\cdot^0(T^*M;\cL)$ the \emph{symplectic homology group in the component of contractible loops with coefficients in $\cL$}, see for example~\cite{Abouzaid-cotangent} for the definition.  

\begin{theorem} \label{thm:vanishing}
Let $M$ be a closed manifold satisfying condition~(C) in the Introduction. 
For any rank $1$ local system $\cC$ on $M$ there exists a local system of coefficients $\cL$ on $L_0T^*M$ such that $\cL|_{M}\simeq \cC$ and such that $SH_\cdot^0(T^*M;\cL)$ vanishes.
\end{theorem} 

%

\begin{proof} We know from~\cite{Abouzaid-cotangent}, see also~\cite{AS,AS-corrigendum} for the orientable case and~\cite{Ritter} for a discussion of local coefficients, that there exists a $\Z/2$-local system $\eta$ on $L_0T^*M$, which restricts to the orientation local system $o_M$ on the space of constant loops on the zero section, and such that 
$$
SH_\cdot^0(T^*M;\cL\otimes\eta)\simeq H_\cdot(L_0M;\cL|_{L_0 M})
$$
for any local system $\cL$ on $L_0T^*M$. In the orientable case, this local system is defined by transgressing the second Stiefel-Whitney class of $M$, i.e. the holonomy along a loop in $L_0T^*M$ is $\pm\mathrm{Id}$ according to whether the pull-back of the second Stiefel-Whitney class of $M$ to $T^*M$ evaluates trivially or not on the torus $S^1\times S^1\to T^*M$ determined by that loop. The definition in the non-orientable case is more involved, see~\cite[Chapter~11]{Abouzaid-cotangent}. 

By Theorem~\ref{thm:vanishing-homology}, let $\tilde \cL$ be a local system on $L_0M$ such that $\tilde \cL|_M\simeq \cC\otimes o_M$ and $H_\cdot(L_0M;\tilde \cL)=0$. Since the natural projection $L_0T^*M\stackrel\pi\longrightarrow L_0M$ is a deformation retract, this local system admits a unique extension $\pi^*\tilde \cL$ to a local system on $L_0T^*M$. We obtain $SH_\cdot^0(T^*M;\pi^*\tilde \cL\otimes\eta)=0$ and $\pi^*\tilde \cL\otimes\eta|_M\simeq \cC\otimes o_M^{\otimes 2}\simeq \cC$, so that $\cL=\pi^*\tilde \cL\otimes \eta$ satisfies the desired requirements.  
\end{proof}

\subsection{Uniruledness} \label{sec:uniruledness}

That the vanishing of symplectic homology (with untwisted coefficients) can be used in order to produce pseudoholomorphic curves in Liouville manifolds was first shown by Anne-Laure Biolley in her groundbreaking thesis~\cite{Biolley-thesis}, see also~\cite{Biolley-arxiv}. Biolley founded a whole new theory of symplectic hyperbolicity, motivated by a remarkable theorem of Bangert~\cite{Bangert}.

\begin{definition} \label{defi:uniruled}
Let $\hat W$ be a Liouville manifold of finite type. We say that $\hat W$ is \emph{uniruled} if, for any choice of compatible almost complex structure $J$ which is cylindrical outside a compact set with respect to some Liouville structure, and for any choice of point $x\in\hat W$, there exists a non-constant finite energy pseudo-holomorphic plane $(\C,i)\to (\hat W,J)$ through $x$. 
\end{definition}

We refer to~\cite{Cieliebak-Eliashberg} for the definition of Liouville manifolds of finite type and of homotopies of Liouville structures. Our point of view here is that a Liouville manifold is a symplectic manifold carrying the additional data of a deformation equivalence class of Liouville structures. For practical purposes we view $\hat W$ as the union $W\, \cup \, [1,\infty)\times\p W$, where $W$ is a Liouville domain and $[1,\infty)\times \p W$ is the positive half of the symplectization of its contact boundary $\p W$. 
A choice of Liouville structure determines a unique such presentation up to deformation, and so does the choice of a deformation equivalence class of Liouville structures. 

In order to place Definition~\ref{defi:uniruled} into context, let us recall that a complex manifold $(X,J)$ is said to be \emph{Brody hyperbolic} if any holomorphic map from the complex plane into $X$ is constant. The most famous example of a Brody hyperbolic complex manifold is $\C P^1$ minus three points, this being the content of the Little Picard Theorem. The classical reference on complex hyperbolic manifolds is~\cite{Lang-hyperbolicity}. Let us call a complex manifold \emph{uniruled} if any point sits on the image of some non-constant map defined on the complex plane. Uniruled complex manifolds are  the opposite of hyperbolic ones as they feature an abundance of non-constant holomorphic planes. 
The definition of uniruledness in the almost complex setting should be seen as a symplectic generalization of complex uniruledness, in the spirit of Biolley's definition of symplectic (non-)hyperbolicity.

\begin{theorem} \label{thm:uniruled}
Let $\hat W$ be a Liouville manifold of finite type and assume that there exists a local system of coefficients $\cL$ on the constant loop component of the free loop space $L_0\hat W$ such that
\begin{enumerate}
\item symplectic homology of $\hat W$ in the component of constant loops with coefficients in $\mathcal L$ vanishes, and 
\item the cohomology of $\hat W$ with coefficients in $\cL|_{\hat W}$ does not vanish. 
\end{enumerate}
Then $\hat W$ is uniruled. 
\end{theorem}

\noindent {\bf Remark.} This result is encompassed by Biolley's~\cite[Theorem~5.1]{Biolley-arxiv}, but we give an alternative proof based on the compactness theorem in symplectic field theory~\cite{BEHWZ}. Indeed, Biolley's theorem involves a condition of controlled capacity growth which is automatically satisfied for Liouville manifolds of finite type, and although its statement involves symplectic homology with untwisted coefficients the proof adapts \emph{verbatim} to the case of twisted coefficients. In particular, Biolley's arguments yield the following, still under the assumptions of the theorem: \emph{given any tame almost complex structure $J$ on $W$, every point sits on the image of a $J$-holomorphic disc in $W$ with boundary on $\p W$}.

\begin{proof}[Proof of Theorem \ref{thm:uniruled}] Assume first that $H^0(\hat W;\cL|_{\hat W})\neq 0$. Pick any point $x\in \hat W$ and a presentation $\hat W=W\, \cup\, [1,\infty)\times\p W$ such that $x\in W$. Let $H:S^1\times \hat W\to\R$ be a periodic time-dependent Hamiltonian with nondegenerate $1$-periodic orbits which satisfies the following conditions: (i) the restriction of $H$ to $W$ is a (time-independent) $C^2$-small Morse function which has a unique minimum at $x$ and whose gradient is outward pointing along $\p W$; (ii) in the cylindrical end $[1,\infty)\times \p W$ the Hamiltonian is a small perturbation of a time-independent convex function of the radial coordinate $r\in[1,\infty)$, which is strictly increasing and becomes linear outside some compact set. (We call such a Hamiltonian \emph{admissible} for $W$.) Note that the unique minimum $x$ represents in Morse homology the fundamental class of the manifold in $H_{2n}(W,\p W;\cL|_W)\simeq H^0(W;\cL|_W)$. Since the twisted symplectic homology vanishes, there exists $\mu>0$ such that, if the maximal slope of the Hamiltonian $H$ is bigger than $\mu$, then the image of the fundamental class $[x]$ in the Floer homology of $H$ is zero. On the other hand, we assumed that this fundamental class is nonzero, which implies that there exists a Floer trajectory from some non-constant $1$-periodic orbit of $H$ to the minimum $x$. 

We now stretch the neck in a small collar neighborhood of $\p W$ while at the same time letting the Hamiltonian go to zero on $W$, similarly to~\cite{Bourgeois_Oancea_sequence}. The outcome of the analysis therein and of the Symplectic Field Theory compactness theorem~\cite{BEHWZ} is that there is a sequence of Floer trajectories as above which converges to a building that solves a Hamiltonian perturbed Cauchy-Riemann equation on the top level, and that contains in the bottom level a $J$-holomorphic plane passing through $x$. 

Note that, as a byproduct of the method, the $J$-holomorphic plane that we find is asymptotic to a closed Reeb orbit at infinity, and it has Hofer energy smaller than $\mu$. 

In the general case where we assume the existence of a nonzero cohomology class in some $H^k(\hat W;\cL|_{\hat W})$ for $k>1$ the proof is similar, except for notational changes that arise from the choice of a sequence of Floer trajectories from some non-constant $1$-periodic orbit of $H$ to one of the critical points of index $k$ of $H$.  
\end{proof}

\noindent {\bf Remark.} Work in progress by Jungsoo Kang and Richard Siefring aims to produce some remarkable finite energy foliations in the $4$-dimensional case using a more refined asymptotic analysis, under the same assumptions. 

\noindent {\bf Remark.} Assume there is a local system $\cL$ as in Theorem~\ref{thm:uniruled} such that $\cL|_{\hat W}$ is trivial, so that we have $H^0(\hat W;\cL|_W)\neq 0$. Given a presentation $\hat W=W\cup[1,\infty)\times\p W$ of the Liouville domain $W$, we define $SH_\cdot^{(-\infty,\mu),0}(W;\cL)$, $\mu>0$ as the direct limit of Floer homology groups $FH_\cdot^{(-\infty,\mu)}(H;\cL)$ truncated in the action window $(-\infty,\mu)$ in the component of constant loops, the limit being taken over increasing families of admissible Hamiltonians for $W$ as above. We then have $H_{\cdot+n}(W,\p W)\simeq SH_\cdot^{(-\infty,\mu),0}(W;\cL)$ for any $\mu>0$ smaller than the minimal period of a closed Reeb orbit on $\p W$. Denote $SH_\cdot^0(\hat W;\cL)$ the symplectic homology of $\hat W$ in the component of constant loops with coefficients in $\cL$, which we assume to vanish. In view of $\lim_{\mu\to\infty}ÊSH_\cdot^{(-\infty,\mu),0}(W;\cL)=SH_\cdot^0(\hat W;\cL)=0$, we infer that 
$$
\mu(W,\cL):=\inf \big\{\mu\mid\mbox{the map }ÊH_{2n}(W,\p W)\to SH_\cdot^{(-\infty,\mu),0}(W;\cL) \mbox{ vanishes}\big\}<\infty.
$$
Viterbo~\cite{Viterbo99} and Biolley~\cite{Biolley-thesis} call $\mu(W,\cL)$ \emph{the capacity of the Liouville domain $W$}. Since $SH_\cdot^{(-\infty,\mu),0}(W;\cL)$ only changes if $\mu$ crosses a period of a closed Reeb orbit on $\partial W$ the capacity $\mu(W,\cL)$ is a period of a closed Reeb orbit. As a matter of fact $\mu(W,\cL)$ is a period of a \emph{contractible} Reeb orbit since the vanishing of the map $H_{2n}(W,\p W)\to SH_\cdot^{(-\infty,\mu),0}(W;\cL)$ means that there is a Floer trajectory from this Reeb orbit to a constant loop, see the proof of Theorem \ref{thm:uniruled}.

\noindent {\bf Remark.} The \emph{Weinstein conjecture}~\cite{Weinstein_The_conjecture} asserts the existence of a periodic orbit for any Reeb vector field on a closed contact manifold. Viterbo proved in \cite{Viterbo99} that vanishing of symplectic homology of a Liouville domain $W$ implies the Weinstein conjecture for separating contact hypersurfaces in $W$. This continues to hold with twisted coefficients that are constant on the subspace of constant loops. Indeed, for the remarkable hypersurface $\p W$ this follows from $\mu(W,\cL)<\infty$. For general hypersurfaces, this follows from Viterbo functoriality \emph{\`a la} Ritter~\cite{Ritter09}. Note that the Reeb orbits that we find in this way are contractible, though there are contact manifolds which are hypertight, meaning that no Reeb vector field has a contractible Reeb orbit. This phenomenon is related to symplectic fillability, see~\cite{MNW13} and the references therein. This stronger conclusion should be compared with Corollary \ref{cor:finite_HZ_cap} which also asserts the existence of a \emph{contractible} orbit of a Hamiltonian system.

The proof of Theorem~\ref{thm:uniruled} shows in particular the following. 

\begin{corollary} Assume that the local system $\cL$ in the statement of Theorem~\ref{thm:uniruled} is trivial on the subspace of constant loops $\hat W$. Given a decomposition $\hat W=W\, \cup\,  [1,\infty)\times \p W$ the uniruling $J$-holomorphic planes can be chosen such that their Hofer energy is bounded by $\mu(W,\cL)$. 
\hfill{$\square$}
\end{corollary}

This echoes a result of Biolley~\cite[Theorem~3.1]{Biolley-arxiv} for pseudo-holomorphic discs, with the Hofer energy replaced by the $L^2$-energy/symplectic area.

Putting together Theorems~\ref{thm:uniruled} and~\ref{thm:vanishing} we obtain:

\begin{theorem}
Let $M$ be a closed manifold satisfying condition~(C) in the Introduction. Then $T^*M$ is uniruled. 
\end{theorem}
 
\begin{proof} Indeed, it is enough to choose in Theorem~\ref{thm:vanishing}  a local system which is trivial on the space of constant loops, so that Theorem~\ref{thm:uniruled} can be applied. 
\end{proof}


\subsection{Finiteness of the Hofer-Zehnder capacity} \label{sec:HZ}

Let $W$ be a connected Liouville domain of dimension $2n$.
%
%
Abbreviate
$$\mathscr{S}:=\mathscr{S}(W):=\bigcup_{\ell\,\,\textrm{prime}}\Hom_{inv}(\pi_1(\Omega_0W),\Z/\ell)$$
which we interpret as local systems on the free loop space of $W$ with fiber $\C$ and holonomy in $\Z/\ell$, and which restrict to trivial local systems on
$W$ itself. We consider the subset
$$\mathscr{S}_0:=\mathscr{S}_0(W):=\{\mathcal{L} \in \mathscr{S}: SH_\cdot^0(W;\mathcal{L})=0\}$$
where $SH_\cdot^0(W;\mathcal{L})$ denotes symplectic homology of the Liouville domain in the component of the constant loops with coefficients twisted
by $\mathcal{L}$. The grading on $SH_\cdot^0$ is only $\Z/2$-valued unless the first Chern class $c_1(W)$ is $2$-torsion, in which case it is $\Z$-valued.  Our aim is to explain that under the assumption that $\mathscr{S}_0 \neq \emptyset$ the $\pi_1$-sensitive
Hofer-Zehnder capacity of $W$ is finite. Examples where $\mathscr{S}_0 \neq \emptyset$ are:
\begin{enumerate}\renewcommand{\theenumi}{\roman{enumi}}\itemsep=1ex
 \item Subcritical Stein domains~\cite{Cieliebak} or more generally Liouville domains which admit a displaceable embedding~\cite{Ritter,Kang14}. In this case $\mathscr{S}_0=\mathscr{S}$ so that already the trivial local system lies in $\mathscr{S}_0$.
 \item Flexible Weinstein domains, see~\cite[\S11.8, \S17.2]{Cieliebak-Eliashberg}.
 \item Fiberwise starshaped domains $W \subset T^*M$ in the cotangent bundle of a closed manifold which satisfies condition~(C) in the Introduction.
 In this case the set $\mathscr{S}_0$ is nonempty and consists of $\mathscr{S}$ minus one element. This is the content of Theorem~\ref{thm:vanishing}.
\end{enumerate} 
We recall that if $SH_\cdot^0(W;\cL)=0$ the capacity of the Liouville domain $W$ with respect to the local system $\cL$ is denoted by  $\mu(W,\cL)$.  To get a quantity which only depends on $W$ we set
$$\mu(W):=\inf\{\mu(W,\mathcal{L}): \mathcal{L} \in \mathscr{S}_0(W)\}.$$
If $\mathscr{S}_0(W)=\emptyset$ we set $\mu(W):=\infty$.

We next describe two Hofer-Zehnder capacities on $W$. We need the following terminology. A smooth function
$H \colon W \to [0,\infty)$ with compact support in the interior of $W$ is called \emph{simple} if it satisfies the following two conditions:
\begin{enumerate}\renewcommand{\theenumi}{\roman{enumi}}
 \item There exists a nonempty open subset $U \subset W$ such that $H|_{U}=\max H$,
 \item The only critical values of $H$ are $0$ and $\max H$.
\end{enumerate}
In particular, since $H$ is nonnegative it follows that the oscillation of $H$ equals
$$\mathrm{osc}(H):=\max H-\min H=\max H.$$
A simple function $H$ is called HZ-\emph{admissible} if the Hamiltonian vector field of $H$ has no nonconstant
periodic orbits of period less than or equal to $1$, and a simple function is called $\textrm{HZ}^0$-\emph{admissible} if the same is true for contractible periodic orbits. The two Hofer-Zehnder capacities of the Liouville domain $W$ are defined as
$$c_{HZ}(W):=\sup\{\max H: H\,\,\textrm{HZ-admissible}\},$$
$$c^0_{HZ}(W):=\sup\{\max H: H\,\,\textrm{HZ}^0\textrm{-admissible}\}.$$
We have the following obvious inequality
$$c_{HZ}(W) \leq c^0_{HZ}(W).$$

As a motivation, let us recall that finiteness of the Hofer-Zehnder capacities has the following striking consequence.


\begin{theorem}[almost existence theorem {\cite[Chapter 4, Theorem 4]{Hofer_Zehnder_Book}}] \label{thm:almost_existence}
Let $H \colon W \to \mathbb{R}$ be such that $0$ is a regular value and $H|_{\partial W}>0$. Given an interval $I$ containing $0$ and consisting of regular values, the following hold:
 
If $c_{HZ}(W)<\infty$, then almost every level set of $H$ in the interval $I$ carries a periodic orbit.

If $c_{HZ}^0(W)<\infty$, then almost every level set of $H$ in the interval $I$ carries a periodic orbit which is contractible in $W$. \hfill{$\square$}
\end{theorem}

The following theorem is due to Irie~\cite[Corollary 3.5]{Irie-HZ}. The statement and the proof in~\cite{Irie-HZ} assume constant coefficients, but they carry over \emph{verbatim} to the case of twisted coefficients. 

\begin{theorem}\label{thm:Irie}
The following inequality holds
$$c_{HZ}^0(W) \leq \mu(W).$$
\hfill{$\square$}
\end{theorem}

Therefore, we obtain the following conclusion.

\begin{corollary}\label{cor:finite_HZ_cap}
If $\mathscr{S}_0(W) \neq \emptyset$ both Hofer-Zehnder capacities
$c_{HZ}(W)$ and $c_{HZ}^0(W)$ are finite and the almost existence theorem holds in $W$. \hfill{$\square$}
\end{corollary}

This of course proves the theorem announced in the Introduction. 

\noindent {\bf Remark.} While $c_{HZ}^0$ need not necessarily be finite for all cotangent disc bundles as shown by the example of the torus discussed in the Introduction, we do not know examples of Liouville domains for which $c_{HZ}$ is infinite. 

%
%

\section{Further directions} \label{sec:open_questions}

We discuss briefly possible further extensions of the ideas presented in this article. 

\begin{enumerate}
\item Which Weinstein domains admit local systems such that the corresponding symplectic homology vanishes? Ritter proved in~\cite{Ritter-ALE} that this is the case for ALE spaces.
\item Discuss local systems on more general path spaces, e.g. fix a submanifold $L\subset W$ and consider the space of paths in $W$ starting and ending on $L$. Using a reasoning similar to the one of the current paper this should give rise to finiteness results for the symmetric Hofer-Zehnder capacity as introduced by Liu-Wang~\cite{Liu_Wang}. Following a general pattern which first appeared in~\cite{Ritter}, this should also lead to existence results for Reeb chords between Legendrian submanifolds, i.e.~applications to Arnold's chord conjecture.
\item Derive obstructions to exact Lagrangian embeddings $L\hookrightarrow W$ into a Liouville domain $W$ in terms of the map $\pi_2(L)\to\pi_2(W)$ using Viterbo functoriality. Such obstructions were first obtained by Ritter in~\cite{Ritter09}.
\end{enumerate}

\bibliographystyle{plain}
\bibliography{Loc_sys}

\end{document}